\documentclass[english]{amsproc}
\usepackage[T2A]{fontenc}
\usepackage[english]{babel}
\usepackage{graphics}
\usepackage{amsfonts, amssymb, amscd}
\usepackage[dvips]{graphicx}
\usepackage{latexsym}

\DeclareMathOperator{\link}{link}
\DeclareMathOperator{\cone}{cone}

\DeclareMathOperator{\conv}{conv}
\DeclareMathOperator{\relint}{relint}
\DeclareMathOperator{\rk}{rk} \DeclareMathOperator{\pyr}{pyr}
\DeclareMathOperator{\Tor}{Tor}

\newcommand{\Zo}{\mathbb{Z}}
\newcommand{\Ro}{\mathbb{R}}
\newcommand{\Co}{\mathbb{C}}

\newcommand{\ko}{\Bbbk}
\newcommand{\Zt}{\Zo_2}

\newcommand{\wh}[1]{{\widehat{#1}}}

\newcommand{\F}{\mathcal{F}}

\newcommand{\Hr}{\tilde{H}}

\newcommand{\Ss}{\mathbb{S}}
\newcommand{\Z}{\mathcal{Z}}
\newcommand{\Zr}{_{\Ro}\mathcal{Z}}
\newcommand{\sr}{s_{\Ro}}
\newcommand{\ts}{\sigma}

\newcounter{stmcounter}[section]

\newcounter{defcounter}[section]

\numberwithin{equation}{section}

\renewcommand{\thestmcounter}{\thesection.\arabic{stmcounter}}

\newcommand{\ex}{\par\vspace{0.5 cm}\noindent\refstepcounter{stmcounter}\textsc{Example \thestmcounter.}\quad}
\newcommand{\rem}{\par\vspace{0.5 cm}\noindent\refstepcounter{stmcounter}\textsc{Remark \thestmcounter.}\quad}

\newtheorem{cor}[stmcounter]{Corollary}

\newtheorem{thm}[stmcounter]{Theorem}
\newtheorem{prop}[stmcounter]{Proposition}
\newtheorem{lemma}[stmcounter]{Lemma}
\newtheorem{defin}[stmcounter]{Definition}

\newtheorem{claim}[stmcounter]{Claim}

\begin{document}

\title{Simplicial complexes Alexander dual to boundaries of polytopes}
\author{Anton Ayzenberg}
\address{Osaka City University}
\email{ayzenberga@gmail.com}

\begin{abstract}
In the paper we treat Gale diagrams in a combinatorial way. The
interpretation allows to describe simplicial complexes which are
Alexander dual to boundaries of simplicial polytopes and, more
generally, to nerve-complexes of general polytopes. This technique
and recent results of N.Yu.Erokhovets are combined to prove the
following: Buchstaber invariant $s(P)$ of a convex polytope equals
1 if and only if $P$ is a pyramid. In general, we describe a
procedure to construct polytopes with $\sr(P)\geqslant k$. The
construction has purely combinatorial consequences. We also apply
Gale duality to the study of bigraded Betti numbers and
$f$-vectors of polytopes.
\end{abstract}

\maketitle

\section{Introduction}\label{SecIntro}

Gale duality is a classical notion in convex geometry. Since its
appearance in \cite{Gale} it allowed to prove many strong and
nontrivial results for convex polytopes and configurations of
points on a sphere (the survey of this field can be found in
\cite{Gr}). In this paper we describe a surprisingly simple
connection between Gale diagrams and combinatorial Alexander
duality.

For any set of points on a sphere $\Ss^r$ we associate a covering
of $\Ss^r$ by hemispheres. From the theory of Gale duality follows
that the nerves of such coverings are exactly those complexes,
which are Alexander dual to boundaries of simplicial polytopes,
or, more generally, to nerve-complexes of polytopes (see claim
\ref{claimGaleAlex} for the precise statement). On one hand, this
gives a combinatorial characterization of complexes dual to
boundaries of simplicial polytopes. On the other hand, geometrical
considerations, involving coverings by hemispheres allowed to
prove particular statements about convex polytopes.

In section \ref{SecSimpConstr} we review and define basic
constructions, used in the work. These include Alexander duality
for simplicial complexes; nerve-complexes of nonsimplicial
polytopes and the construction of a \emph{constellation complex}
for a configuration of points on a sphere. In section
\ref{SecConstel} are listed the most important topological and
combinatorial properties of constellation complexes. In section
\ref{SecGale} the Gale duality is applied to show that
constellation complexes are Alexander dual to nerve-complexes of
polytopes. Alexander duality allows to simplify and treat
topologically many well known results.

In section \ref{SecBetti} we provide basic definitions from
commutative algebra. Arguments, similar to those used by Eagon and
Reiner in \cite{EaRe} are applied to constellation complexes. We
use Hochster formula to show that the Stanley--Reisner ideal of a
constellation complex $\Delta(X)$ has a linear resolution. This
means that all generators of modules in the minimal resolution are
concentrated in prescribed degrees. Alexander duality leads to the
following result: if $X$ is a Gale diagram of a simplicial
polytope $P$, then bigraded Betti numbers of the constellation
complex $\Delta(X)$ coincide with the $f$-vector of $P$
(proposition \ref{propBettiConstel}). This correspondence can be
naturally generalized to polytopes which are not simplicial
(proposition \ref{propBettiConstelGen}). On the other hand, one
can calculate bigraded Betti numbers of a polytope $P$ by studying
the combinatorial topology of $\Delta(X)$.

Originally, this research was motivated by the study of Buchstaber
invariant (definition \ref{definBuchNumber}). This invariant of
simplicial complexes and polytopes appeared naturally in toric
topology, and in 2002 V.M.Buchstaber posed a problem: to describe
this number combinatorially.

Since then several approaches to this problem had been developed.
I.Izmestiev \cite{Izm1,Izm2} found a connection between Buchstaber
invariants and a chromatic number. This connection was further
developed in work \cite{Ayzs}. In the paper \cite{FM} real
Buchstaber invariant of skeleta of simplices was determined by
integer linear programming. N.Yu.Erokhovets in his thesis
\cite{ErThes}, and other works \cite{Er,ErArx,ErNew,ErNewBig}
built the theory of Buchstaber invariants, constructed many
examples and estimations, and found equivalent definitions for
these numbers. We refer the reader to his surveys \cite{ErArx} or
\cite{ErNewBig} to find out more about particular results and open
problems in this field.

Recent result \cite{ErNew} and Gale duality allowed to prove the
following conjecture, made in \cite{ABarx}. If $P$ is a polytope,
then $s(P)=1$ if and only if $P$ is a pyramid (theorem
\ref{thmPyramidS}). Result of \cite{ErNew} can also be applied to
construct polytopes with $\sr(P)\geqslant k$ from their Gale
diagrams. This consideration had led to interesting combinatorial
consequences (theorem \ref{propFanoCircle} and statement
\ref{propFanoCircle}).

We hope that Gale diagrams will allow to answer the question posed
by Nickolai Erokhovets in \cite{ErNewBig}: whether Buchstaber
invariant is determined by bigraded Betti numbers of a simplicial
complex? Probably, Gale diagrams will lead to the solutions of
other open problems concerning Buchstaber invariants of polytopes.

The author is grateful to professor V.\,M.\,Buchstaber for his
suggestion to consider the invariant $s$ on the class of nonsimple
polytopes and for his interest to this work. Also I wish to thank
the participants of the student geometry and topology seminar in
Moscow State University. A few talks made at this seminar turned
out to be very useful for the understanding of topics mentioned in
this paper.

\section{Simplicial complexes from polytopes and spherical configurations}\label{SecSimpConstr}

\subsection{Simplicial complexes}

Let $K$ be a simplicial complex on a set of vertices
$[m]=\{1,\ldots,m\}$. In the following the complex and its
geometrical realization are denoted by the same letter for the
sake of simplicity. By $\Delta_{[m]}$ or $\Delta^{m-1}$ we denote
the simplex on a set $[m]$.

Suppose, $K\neq \Delta_{[m]}$. In this case the \emph{dual
complex} $\wh{K}$ (or $K^{\wedge}$) is a complex on a set $[m]$
defined by
$$
\wh{K}=\{I\in [m]\mid [m]\setminus I\notin K\}.
$$
Obviously, double dual $K^{\wedge\wedge}$ coincides with $K$. In
the literature (e.g. \cite{BP}) this duality is also called
combinatorial Alexander duality, since both $K$ and $\wh{K}$ can
be embedded in barycentric subdivision
$(\partial\Delta_{[m]})'\cong S^{m-2}$ as Alexander dual
subcomplexes (see \cite[sec. 2.4]{BPnew}). As a consequence,
\begin{equation}\label{eqAlexHomol}
\Hr_i(K;\ko)\cong\Hr^{m-3-i}(\wh{K};\ko),
\end{equation}
where $\ko$ is a field or $\Zo$.

For a simplicial complex $K$ on a set $[m]$ the following notions
and notation will be used in the paper:
\begin{itemize}
\item If $J\subseteq [m]$, then $K_J$ is a \emph{full subcomplex} on a
set $J$. Its simplices are those simplices of $K$ which are
subsets of $J$.

\item If $I\in K$ is a simplex, then its \emph{link} is a complex on a set
$[m]\setminus I$ defined by $\link_KI=\{J\subseteq [m] \setminus
I\mid I\sqcup J\in K\}$.

\item If $i\in [m]$, but $\{i\}\notin K$, then $i$ is called the
\emph{ghost vertex} of $K$. Ghost vertices do not affect the
geometry of simplicial complex but they make combinatorial
reasoning simpler. In particular, by definition, links usually
have many ghost vertices.

\item A set $I\subseteq [m]$ is called a \emph{minimal nonsimplex} (it is
also called a missing face in the literature) if $I\notin K$, but
any proper subset of $I$ is a simplex. The set of all minimal
nonsimplices of $K$ will be denoted by $N(K)$.

\item $K^{(l)}$ denotes $l$-dimensional skeleton of $K$.

\item $\Hr(\varnothing;\ko)\cong \ko$, where $\varnothing$ is a complex,
which do not have nonempty simplices.
\end{itemize}

One can show that
\begin{equation}\label{eqLinkScAlex}
(\link_KI)^{\wedge} = \wh{K}_{[m]\setminus I} \mbox{ and }
(K_J)^{\wedge} = \link_{\wh{K}}([m]\setminus J),
\end{equation}
when $I\in K$ and $J\notin K$. Note, that these conditions imply
$[m]\setminus I \notin \wh{K}$ and $[m]\setminus J\in \wh{K}$,
which makes all objects well-defined.

\subsection{Polytopes and nerve-complexes}

Now consider a convex polytope $P\subset \Ro^d$, given as a convex
hull of its vertices $P=\conv\{y_1,\ldots,y_m\}$. Suppose $\dim P
= d$. For such a polytope construct an abstract simplicial complex
$K(P)$ on a set $[m]$. Its simplices are those subsets
$I=\{i_1,\ldots,i_k\}\subseteq [m]$ for which corresponding
vertices $y_{i_1},\ldots,y_{i_k}$ belong to a common facet of $P$.

Obviously, if $P$ is simplicial, then $K(P)$ coincides with the
boundary $\partial P$. In general, when $P$ is not simplicial,
$K(P)$ has more complicated structure --- it is not a simplicial
sphere and it can be non-pure (see fig. \ref{pictPrismNerve}).
Such complexes were called \emph{nerve-complexes} and their
properties were described in \cite{AB}.

\ex Let $P$ be a triangular prism. The complex $K(P)$ and its
maximal simplices are illustrated in fig. \ref{pictPrismNerve}

\begin{figure}[h]
\begin{center}
\includegraphics[scale=0.2]{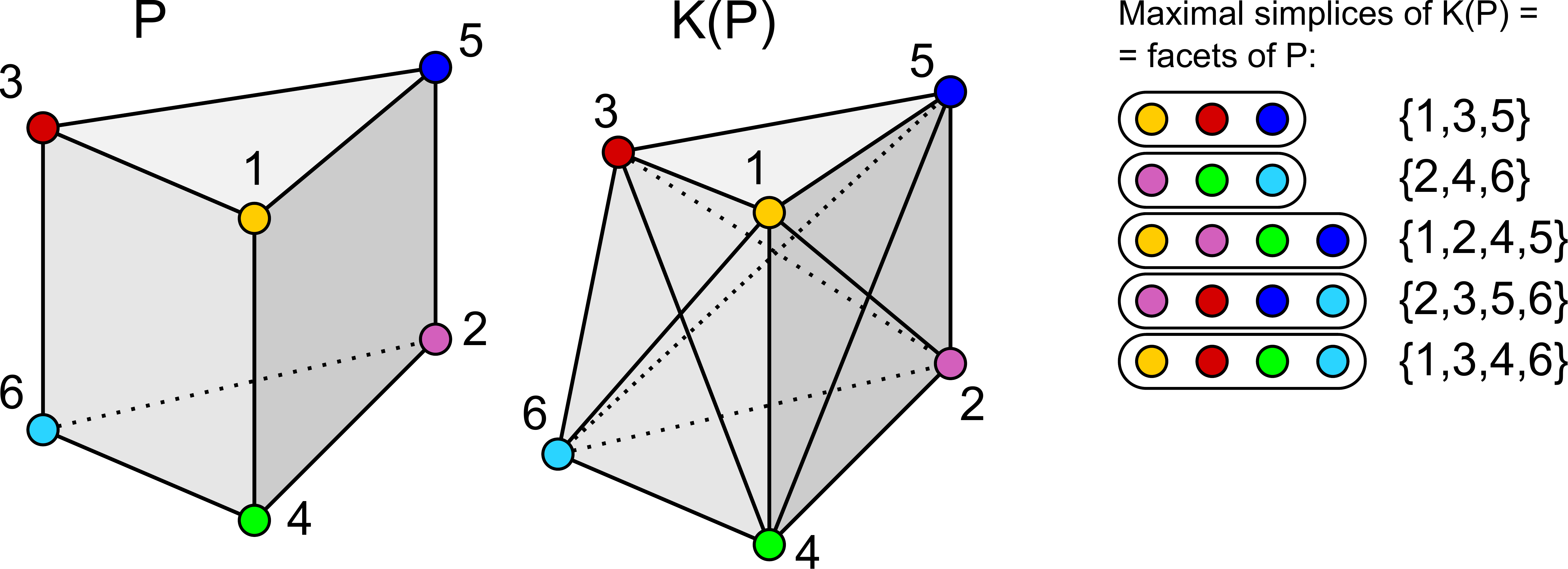}
\end{center}
\caption{Complex $K(P)$ for a triangular
prism.}\label{pictPrismNerve}
\end{figure}

Another way to define nerve-complexes (which explains their name)
is the following. Consider a polytope $Q$ with facets
$\F_1,\ldots,\F_m$. Define a simplicial complex $K_Q$ on a set
$[m]$ as a nerve of the cover $\partial Q = \bigcup\F_i$. In other
words, $\{i_1,\ldots,i_k\}\in K_Q$ iff facets
$\F_{i_1},\ldots,\F_{i_k}$ intersect. Then, $K(P^*) = K_{P}$,
where $P^*$ is a polar dual polytope to $P$. In \cite{AB} we used
$K_Q$-construction rather than $K(P)$-construction for toric
topology reasons, but the outcome of these constructions is the
same. In particular, it is proved that $K(P)$ defines the face
lattice of a polytope $P$ uniquely. So far there is no loss of
combinatorial information when $K(P)$ is considered instead $P$.

\subsection{Sphere diagrams and constellation complexes}

Let $\Ss^r$ be a unit sphere in euclidian space: $\Ss^r=\{x\in
\Ro^{r+1}\mid |x|=1\}$. The notation $\Ss$ is reserved for
geometrical object; the letter $S$ is used for sphere as a
topological space or a homotopy type.

For a point $x\in \Ss^r\sqcup\{0\}$ define a subset $H(x)=\{y\in
\Ss^r\mid \langle x,y\rangle>0\}$. If $x=0$, then $H(x)$ is empty.
If $x\in \Ss^r$, the set $H(x)$ is the open hemisphere,
corresponding to $x$.

Let $X=\{x_1,\ldots,x_m\}$ be a collection of points, $x_i\in
\Ss^r\sqcup \{0\}$ (repetitions are allowed, so $X$ is a
multiset). Such configurations also appear as \emph{spherical
codes} in the literature in connection with geometrical
optimization problems (e.g. \cite{CS}). Consider the covering
$\bigcup_iH(x_i)$.

\begin{defin}[Constellation complex]
The nerve $\Delta(X)$ of this covering will be called the
constellation complex of a configuration $X$. It means that
$\Delta(X)$ is a simplicial complex on a set $[m]$, and
$\{i_1,\ldots,i_k\}\in \Delta(X)$ iff $H(x_{i_1}) \cap\ldots\cap
H(x_{i_k})\neq \varnothing$.
\end{defin}

\rem The condition $H(x_{i_1}) \cap\ldots\cap H(x_{i_k})\neq
\varnothing$ means that there exist $y\in \Ss^r$ such that
$\langle y, x_{i_t}\rangle>0$. So far this is equivalent to
$x_{i_1},\ldots, x_{i_k}\in H(y)$. Therefore, the set of indices
$\{i_1,\ldots,i_k\}$ forms a simplex iff corresponding points
$x_{i_1},\ldots, x_{i_k}$ lie in a common open hemisphere. This
explains the terminology: it seems reasonable to call a set of
stars on a celestial sphere a \emph{constellation} if they can be
observed from some point on earth at the same time. Similar
considerations and comparison also appeared in \cite{Marc}.

\rem If $x_i=0$, the vertex $i$ is the ghost vertex of
$\Delta(X)$.

\ex A few examples of constellation complexes for points on
$\Ss^2$ are represented in fig.\ref{pict2sphere}. In the last
image we took $4$ points, which contain $0$ in their convex hull.
In this case any three hemispheres intersect, but not four:
$\Delta(X) = \partial \Delta^3$.

\begin{figure}[h]
\begin{center}
\includegraphics[scale=0.2]{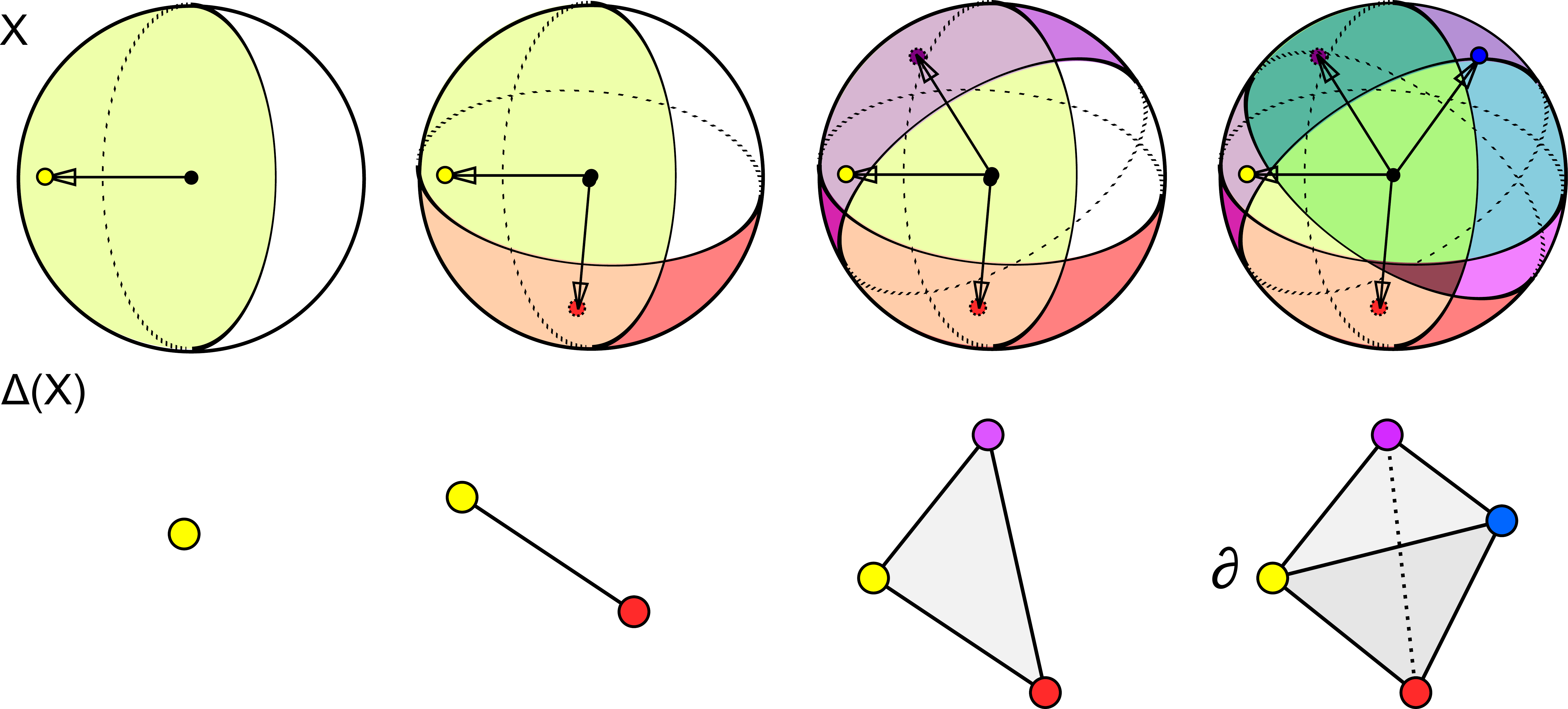}
\end{center}
\caption{Examples of simplest constellation
complexes.}\label{pict2sphere}
\end{figure}

In the following we suppose that $\bigcup_iH(x_i) = \Ss^r$, so
every point on a sphere is covered by some hemisphere. In this
case $\Delta(X)$ is homotopy equivalent to $S^{r}$, since the
nonempty intersections in the covering are contractible (they are
given as intersections of open cones with a sphere). In
particular, this implies $\langle X\rangle=\Ro^{r+1}$.

For a subset of labels $I=\{i_1,\ldots, i_k\}$ we denote the
(multi)set of points $\{x_{i_1},\ldots,x_{i_k}\}$ by $X(I)$.

\begin{claim}\label{claimCaratheodory}
Let $X\subset \Ss^r\sqcup \{0\}$. Subset $I$ is a nonsimplex of
$\Delta(X)$ if and only if the points $X(I)$ contain $0$ in their
convex hull. The cardinality of any minimal nonsimplex does not
exceed $r+2$.
\end{claim}

The first statement follows from standard separation arguments in
convex geometry. The second one is Caratheodory's theorem (see,
e.g. \cite[Sec.2.3]{Gr}).

\ex Let $X_5$ be a configuration of $5$ points on a circle $\Ss^1$
placed in vertices of a regular pentagon (fig. \ref{pictCircle5}).
The colored arcs on the left image show the open hemispheres
corresponding to points. In this example $\Delta(X_5)$ is a
M\"{o}bius band --- any 3 consecutive points form a simplex. Claim
\ref{claimCaratheodory} is illustrated by the list of minimal
nonsimplices of $\Delta(X_5)$.

\begin{figure}[h]
\begin{center}
\includegraphics[scale=0.2]{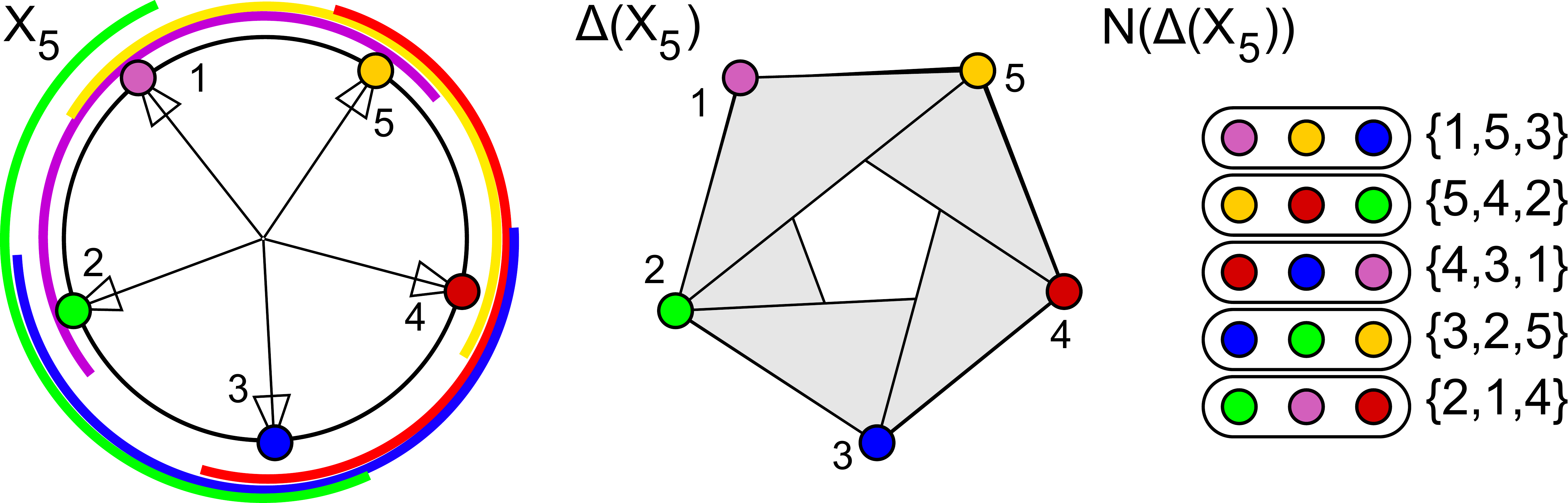}
\end{center}
\caption{Complex $\Delta(X_5)$ for 5 points on a
circle.}\label{pictCircle5}
\end{figure}

\ex Similarly, for $6$ points on $\Ss^1$ placed in vertices of a
regular hexagon we have fig. \ref{pictCircle6}.

\begin{figure}[h]
\begin{center}
\includegraphics[scale=0.2]{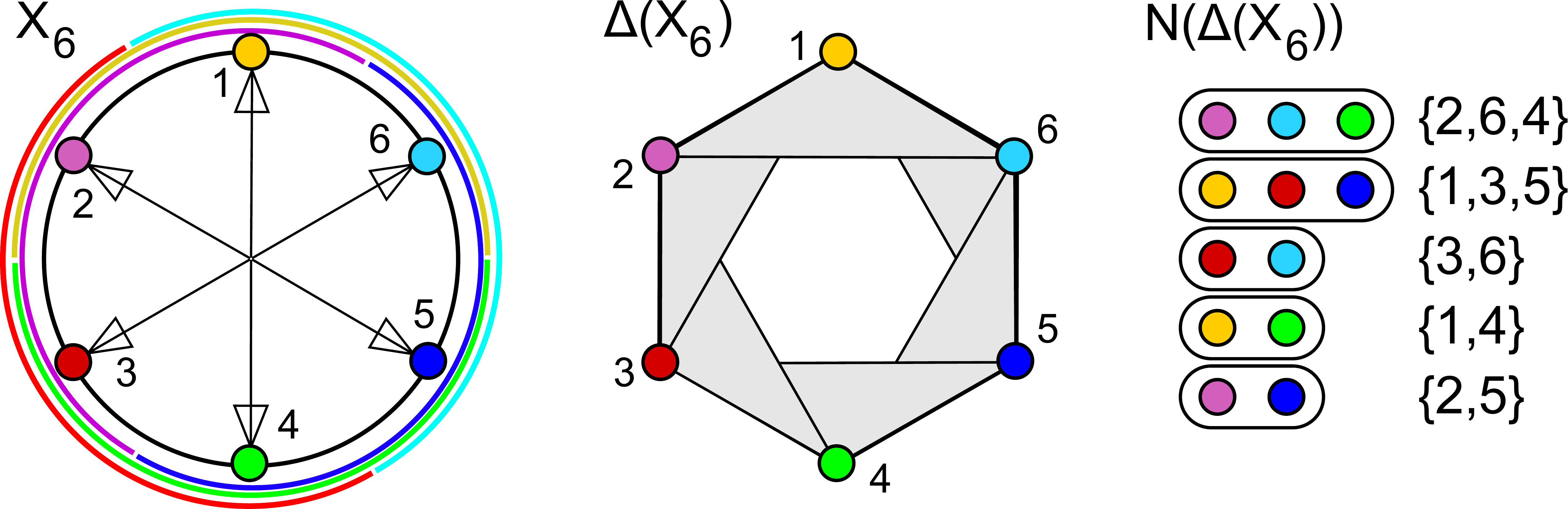}
\end{center}
\caption{Complex $\Delta(X_6)$ for 6 points on a
circle.}\label{pictCircle6}
\end{figure}

The configuration $X$ is called \emph{nondegenerate} if for each
$J$ such that $0\in \conv X(J)$ there holds $\dim \langle
X(J)\rangle = \dim \langle X\rangle = r+1$. Equivalently, $X$ is
nondegenerate if all minimal nonsimplices of $\Delta(X)$ have
cardinality exactly $r+2$.

The configuration $X_5$ on fig.\ref{pictCircle5} is nondegenerate.
Configuration $X_6$ (fig. \ref{pictCircle6}) is degenerate.

\section{Properties of constellation complexes}\label{SecConstel}

Let $K$ be a complex on a set $V$ and $L\subset K$ --- its
subcomplex. Consider a new simplicial complex $K\cup_L\cone L$ on
a set $V\sqcup \{w\}$ --- the result of attaching a cone with apex
$w$ to $K$ along $L$. The operation $K \mapsto K\cup_L\cone L$
will be called a \emph{construction step} in the case when $L$ is
a contractible space.

Any complex $K$ can be decomposed as $K_{[m]\setminus
w}\cup_{\link w}\cone\link w$ for any vertex $w$. The possibility
to make $K$ from $K_{[m]\setminus w}$ by construction step means
that $\link w$ is contractible.

\begin{prop}\label{propConstelProperties}
Let $X\subset \Ss^r$ be a nondegenerate configuration, such that
$\bigcup_{x\in X}H(x)=\Ss^r$, and $\Delta(X)$ --- its
constellation complex. Then

\begin{enumerate}

\item For each subset $I\subseteq [m]$ the full subcomplex
$\Delta(X)_I$ is either a simplex or homotopy equivalent to a
sphere $S^{r}$. If $\Delta(X)_I$ is homotopy equivalent to a
sphere, then so is $\Delta(X)_J$ for $J\supset I$.

\item $\Delta(X)$ can be obtained from $\partial\Delta^{r+1}$ by a
sequence of construction steps.

\end{enumerate}

\end{prop}

\begin{proof}
Note that a full subcomplex $\Delta(X)_I$ coincides with
$\Delta(X(I))$ --- the constellation complex of the smaller set.
To prove (1) consider two possibilities: $0\in \conv X(I)$ or
$0\notin\conv X(I)$. In the first case we actually have $0\in
\relint\conv X(I)$ because of nondegeneracy condition. Therefore,
open hemispheres of $X(I)$ cover $\Ss^r$ and $\Delta(X(I))$ is
homotopy equivalent to $S^r$ by the nerve theorem. In the second
case, when $0\notin\conv X(I)$, the set $X(I)$ is covered by an
open hemisphere, therefore $I\in \Delta(X)$, so
$\Delta(X)_I=\Delta_{I}$.

To prove (2) we proceed as follows. At first, find $J\subseteq
[m]$, such that $0\in\conv X(J)$ and $|J|=r+2$. It exists by
Caratheodory theorem and gives a minimal nonsimplex $J\in
N(\Delta(X))$. Therefore, $\Delta(X(J)) = \partial \Delta^{r+1}$.

Now let $I\supseteq J$. The complex $\Delta(X(I\sqcup \{w\}))$ is
obtained from $\Delta(X(I))$ by attaching a cone along
$\link_{\Delta(X(I\sqcup \{w\}))}w$. This link is given by all
$\{i_1,\ldots,i_k\}\subseteq I$ such that $H(x_w)\cap
H(x_{i_1})\cap\ldots\cap H(x_{i_k})\neq \varnothing$. This means
$\link_{\Delta(X(I\sqcup \{w\}))}w$ coincides with the nerve of
the covering of $H(x_w)$ by hemispheres corresponding to $X(J)$.
This nerve is contractible, since $H(x_w)$ is contractible.
Therefore, $\Delta(X(I\sqcup \{w\}))$ is obtained from
$\Delta(X(I))$ by construction step. Applying this operation
several times allows to build $\Delta(X)$ from
$\Delta(X)_J=\partial\Delta^{r+1}$ by a sequence of construction
steps.
\end{proof}

\rem The proof shows that there is, actually, a variety of ways to
build $\Delta(X)$ from $\partial\Delta^{r+1}$. Once the initial
nonsimplex $J$ is installed other vertices can be added in any
order.

\rem The last part of the proof works well in a more general
situation. Consider a covering $M$ of a sphere $\Ss^r$ by ``hats''
of the form $H(x,\alpha)=\{y\in \Ss^r\mid \langle
x,y\rangle>\alpha\}$ for real $\alpha$ between $0$ and $1$. If
some smaller set $N\subset M$ of hats also covers $\Ss^r$, then
the nerve of $M$ is built from the nerve of $N$ by a sequence of
construction steps. An interesting question is how complicated
could be the starting nerve $N$, say, for fixed $\alpha$.

\rem Things become more complicated for degenerate configurations.
For example, consider a configuration $X\subset\Ss^2$ from
fig.\ref{pictDegenS2}. Points $1,2$ are north and south pole
respectively. Points $3,4,5$ lie on equatorial circle and $0\in
\relint\conv\{3,4,5\}$. In this case $H(3)\cup H(4)\cup H(5) =
\Ss^2\setminus\{1,2\}$. Therefore, $\Delta(X)_{\{3,4,5\}}\simeq
S^1 \nsim S^2$. So in degenerate case may appear subcomplexes
which are not homotopy equivalent to $S^{r}$.

\begin{figure}[h]
\begin{center}
\includegraphics[scale=0.2]{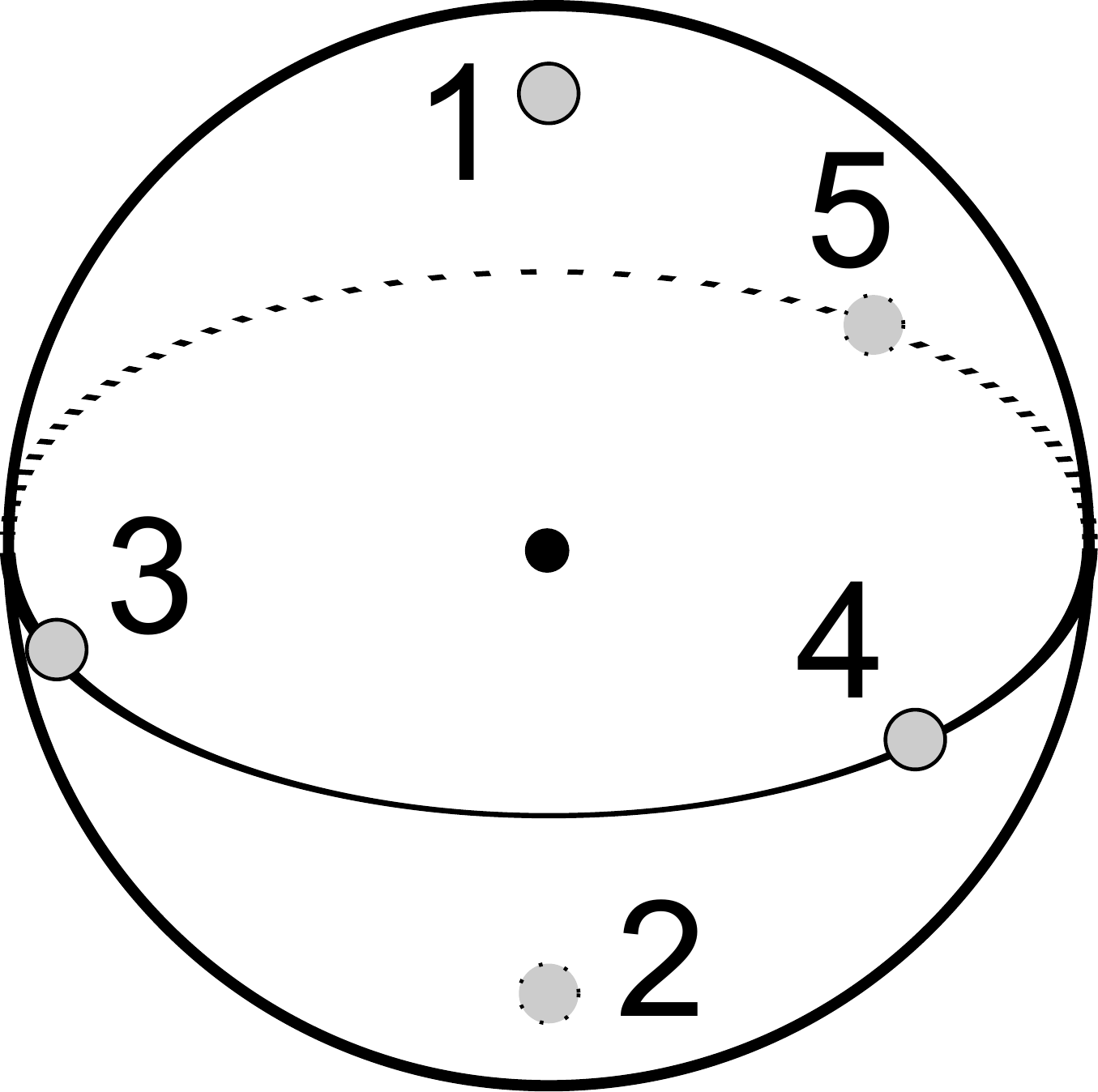}
\end{center}
\caption{Degenerate configuration of 5 points on
$\Ss^2$}\label{pictDegenS2}
\end{figure}

\section{Gale duality described combinatorially}\label{SecGale}

Two objects, defined in section \ref{SecSimpConstr} ---
nerve-complexes and constellation complexes are strongly related.
Roughly speaking, they are Alexander dual to each other.

Let $X = \{x_1,\ldots,x_m\}$ be such a configuration of points on
$\Ss^r$ that any point $x\in \Ss^r$ lies in at least two open
hemispheres $H(x_i)$. So the covering $\bigcup_iH(x_i)$ wraps the
sphere at least twice. In this case configuration $X$ will be
called \emph{good}.

\begin{claim}\label{claimGaleAlex}\mbox{}

\begin{enumerate}
\item A complex $K$ on $m$ vertices is isomorphic to $K(P)$ for
$d$-dimensional polytope $P$ if and only if its Alexander dual
$\wh{K}$ is a constellation complex $\Delta(X)$ for a good
configuration $X\subset\Ss^{m-2-d}\sqcup \{0\}$. Such $X$ can be
constructed as affine Gale diagram $G(P)$ of a polytope $P$.

\item A complex $K$ on $m$ vertices is a boundary of a simplicial
$d$-dimensional polytope if and only if its Alexander dual
$\wh{K}$ is a constellation complex $\Delta(X)$ for a good
nondegenerate configuration $X$ on a sphere $\Ss^{m-d-2}$.

\end{enumerate}

\end{claim}

\begin{proof}
This follows directly from the properties of affine Gale diagrams.
If $Y=\{y_1,\ldots,y_m\}\in \Ro^d$ --- the set of points, then one
can construct its Gale diagram, which is a configuration of
points: $G(Y)=X=\{x_1,\ldots,x_m\}\subset S^{m-2-d}\sqcup\{0\}$.
We refer to \cite[Sec.5.4]{Gr} for the definition of this
construction and its properties.

Let $P$ be a polytope with the set of vertices $Y =
\{y_1,\ldots,y_m\}\subset \Ro^d$, so $P=\conv Y$ and suppose $\dim
P = d$. We assert that $K(P)^{\wedge} = \Delta(G(Y))$. Indeed, let
$I\in K(P)^{\wedge}$. Equiv., $[m]\setminus I\notin K(P)$. Equiv.,
$Y([m]\setminus I)$ is not a subset of a proper face in $P$.
Equiv., $I$ does not contain a coface of $P$ (coface = complement
to the set of vertices of a face). Equiv., by Gale duality, $I$
does not contain a subset $J$ such that $0\in \relint \conv X(J)$.
Equiv., $0\notin \conv X(I)$. Equiv., $I\in \Delta(X)$.

The fact that resulting configuration $X$ should be good and the
second statement of the claim follow from the properties of Gale
diagrams, listed in \cite{Gr}.
\end{proof}

\ex\label{ex6prism} By comparing figures \ref{pictPrismNerve} and
\ref{pictCircle6} one can see that $K(P)$ from the first picture
is Alexander dual to $\Delta(X_6)$ from the second. Indeed,
maximal simplices of a complex are the complements to minimal
nonsimplices of its dual. Equivalent way of saying this: $X_6$ is
a combinatorial Gale diagram for a triangular prism.

\ex The M\"{o}bius band $\Delta(X_5)$ from fig.\ref{pictCircle5}
is Alexander dual to the boundary of a pentagon. Configuration
$X_5$ is a Gale diagram of a pentagon.
\\
\\
The following proposition is well known in the theory of Gale
duality. As before, $Y=\{y_1,\ldots,y_m\}$ --- the set of vertices
of $P=\conv Y$ and $X = \{x_1,\ldots,x_m\}\subset
\Ss^{m-d-2}\sqcup\{0\}$ is a Gale diagram of $Y$ denoted by $G(Y)$
or $G(P)$ for short.

\begin{prop}\label{propPyrZero}
\mbox{}
\begin{enumerate}
\item  If $P = \triangle^{m-1}$, then $G(P) =
\{0,\ldots,0\}\subset\Ro^0$.

\item $P$ is a pyramid with apex $y_i$ if and only if $x_i=0$.

\item More generally, $P$ decomposes as a join $P_1\ast P_2$, if and only if $G(P)$
decomposes as direct sum of $G(P_1)$ and $G(P_2)$.

\end{enumerate}
\end{prop}

\ex The Gale diagram of a square is the multiset
$\{-1,-1,1,1\}\subset \Ss^0$. The Gale diagram of a pyramid over a
square is $\{-1,-1,1,1,0\}\subset \Ss^0\sqcup\{0\}$.
\\
\\
The following proposition is also well known (original paper of
Gale \cite{Gale} or \cite[Cor.5]{Marc}). We provide a simple
reformulation in terms of Alexander duality. Recall, that a
polytope $P$ is called $k$\emph{-neighborly} if any $k$ of its
vertices belong to a common proper face of $P$.

\begin{prop}
Let $P$ be a polytope with $m$ vertices and $X$ --- its Gale
diagram. Then $P$ is $k$-neighborly if and only if
$\dim\Delta(X)\leqslant m-k-2$.
\end{prop}

\begin{proof}
The condition for $P$ to be $k$-neighborly is equivalent to
$\Delta_{[m]}^{(k-1)}\subseteq K(P)$. Applying Alexander duality
to this inclusion gives $K(P)^{\wedge}\subseteq
\left(\Delta_{[m]}^{(k-1)}\right)^{\wedge}$ which can be rewritten
as $\Delta(X)\subseteq \Delta_{[m]}^{(m-k-2)}$. The last statement
is equivalent to $\dim\Delta(X)\leqslant m-k-2$.
\end{proof}

\rem Here we do not assume $P$ is simplicial. The statement holds
in general.

\rem Configurations $X$ of points on a sphere, for which
$\dim\Delta(X)\leqslant m-k-2$ were called \emph{positive
$(k+1)$-spanning sets} in \cite{Marc}. The notion reflects the
fact that after deleting any subset $Y\subset X$ with $|Y|=k$ the
remaining set $X\setminus Y$ contains $0$ in the interior of
convex hull. This means vectors $X\setminus Y$ span $\Ro^{r+1}$
with positive coefficients.
\\
\\
Recall that a simplicial polytope $P$ is called flag if any set of
pairwise connected vertices is a face. In other words, $P$ is flag
if $|J|=2$ for any minimal nonsimplex $J\in N(\partial P)$.

\begin{lemma}\label{lemmaFlag}
If $X$ is the Gale diagram of a simplicial polytope $P$ with $m$
vertices, then $P$ is flag if and only if $|I|=m-2$ for any
maximal simplex $I\in \Delta(X)$.
\end{lemma}

This follows from the fact that any maximal simplex of
$\Delta(X)=(\partial P)^{\wedge}$ is the complement to a minimal
nonsimplex of $\partial P$.

\begin{prop}
Suppose $X\subset \Ss^{r}$ is a nondegenerate good configuration
of points and $m=|X|>2(r+2)$. Then there is a maximal simplex of
$\Delta(X)$ which has less than $m-2$ points.
\end{prop}

\begin{proof}
Suppose the contrary. Then by lemma \ref{lemmaFlag} $X$ is a Gale
diagram for a flag simplicial polytope $P$ with $m$ vertices and
$\dim P = m-r-2>\frac{m}{2}$. It is known that any flag polytope
$P$ has at least $2\dim P$ vertices (see e.g. \cite[lemma
2.1.14]{Gal}) --- the contradiction.
\end{proof}

Another important example of using Gale diagrams is the iterated
simplicial wedge construction.

\ex In the work \cite{BBCGit} an iterated wedge operation was
defined. If $P$ is a simple $d$-polytope with $m$ facets, then a
new polytope $P(j_1,\ldots,j_m)$ is defined, which has $\sum_i
j_i$ facets and dimension $n-m+\sum_ij_i$. The corresponding
operation for simplicial complexes $\partial P^*\mapsto \partial
P^*(j_1,\ldots,j_m)$ can be described combinatorially in several
different ways (see \cite{BBCGit,Ayzit}). The Gale diagram
$G(P^*(j_1,\ldots,j_m))$ can be constructed from $G(P^*)$ by
assigning multiplicities $j_i$ to points $x_i \in G(P^*)$.

\begin{figure}[h]
\begin{center}
\includegraphics[scale=0.2]{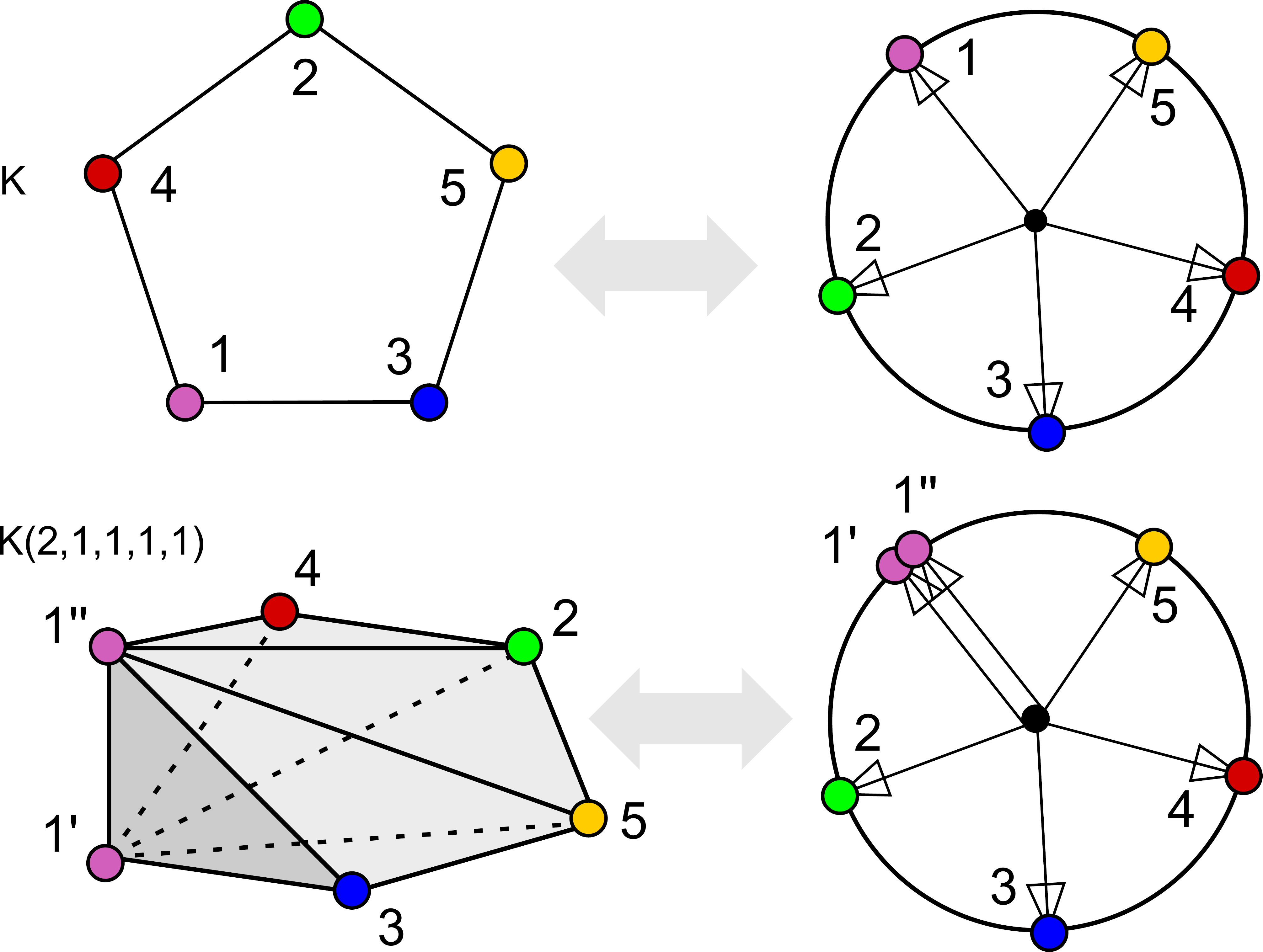}
\end{center}
\caption{Simplicial wedge construction and its effect on Gale
diagram.}\label{pictPentDoubling}
\end{figure}

\ex Fig.\ref{pictPentDoubling} shows how Gale diagram changes when
a simplicial wedge construction is applied to the boundary of a
simplicial polytope. When vertex $1$ is ``wedged'' the
corresponding vertex in Gale diagram doubles.

\section{Bigraded Betti numbers}\label{SecBetti}

Let $\ko$ be a ground field and $\ko[m] = \ko[v_1,\ldots,v_m]$
--- the ring of polynomials in $m$ variables. The ring $k[m]$
is graded by $\deg v_i = 2$. The field $\ko$ is given the
$\ko[m]$-module structure by the epimorphism $\ko[m]\to \ko$,
$v_i\mapsto 0$.

Let $K$ be a simplicial complex on $m$ vertices. The
\emph{Stanley--Reisner algebra} $\ko[K]$ is defined as a quotient
algebra $\ko[m]/I_{SR}$, where the Stanley--Reisner ideal $I_{SR}$
is generated by square-free monomials $v_{\alpha_1}\ldots
v_{\alpha_k}$ corresponding to nonsimplices
$\{\alpha_1,\ldots,\alpha_m\}\notin K$. An algebra $\ko[K]$ is
graded and it has a natural $\ko[m]$-module structure.

Let $\ldots \to R^{-i}\to R^{-i+1}\to\ldots\to R^{-1}\to R^{0}\to
\ko[K]$ be a free resolution of the module $\ko[K]$ by graded
$\ko[m]$-modules $R^{-i}$. We have $R^{-i} = \bigoplus_{j\in
\Zo_{\geqslant}} R^{-i,j}$. The Tor-module of a complex $K$
therefore has a natural double grading:
$$
\Tor_{\ko[m]}(\ko[K],\ko) = \bigoplus\limits_{i,j\in
\Zo_{\geqslant}} \Tor^{-i,2j}_{\ko[m]}(\ko[K],\ko).
$$
The bigraded Betti numbers of a complex $K$ are defined as the
dimensions of the graded components of the Tor-module:
$$
\beta_{\ko}^{-i,2j}(K) = \dim_{\ko}
\Tor^{-i,2j}_{\ko[m]}(\ko[K],\ko).
$$
These numbers depend on a field $\ko$ but we will omit $\ko$ to
simplify notation. The number
$\beta^{-i}(K)=\sum_j\beta^{-i,2j}(K)$ is the rank of a module
$R^{-i}$ in a minimal resolution. Note that $\beta^{0,0}(K)=1$ and
$\beta^{0,2j}(K)=0$ for $j\neq 0$.

Numbers $\beta^{-i,2j}$ represent a lot of combinatorial,
topological and algebraical information about simplicial complex
(see e.g. \cite{BPnew}).

By Hochster formula \cite{Hoch}, \cite[Th.3.2.8]{BPnew}, bigraded
Betti numbers can be expressed in terms of ordinary Betti numbers
of full subcomplexes in $K$:
\begin{equation}\label{equatHochster}
\beta^{-i,2j}(K) = \sum\limits_{J\subseteq [m], |J|=j}\dim
\Hr^{j-i-1}(K_J;\ko),
\end{equation}

\subsection{Bigraded Betti numbers of constellation complexes}

The result of \S\ref{SecConstel} can be stated in terms of Betti
numbers.

\begin{prop}\label{propResolution}
Let $X\subset \Ss^r$ be as in statement
\ref{propConstelProperties}. Then $\beta^{-i,2j}(\Delta(X))=0$ for
$i>0$ and $j\neq r+i+1$. In other words, each module $R^{-i}$ of
the minimal resolution for $\ko[\Delta(X)]$ is generated in degree
$r+i+1$ for $i>0$.
\end{prop}

This directly follows from statement \ref{propConstelProperties}
and formula \eqref{equatHochster}. Here is another equivalent
statement: the Stanley--Reisner ideal $I_{SR}(K)$ has a linear
resolution as a graded module over $\ko[m]$. It means that all
maps in the minimal resolution (which are $\ko[m]$-linear maps)
are linear in variables $v_i$ (see \cite{EaRe} for the details).

Using claim \ref{claimGaleAlex} we can interpret statement
\ref{propConstelProperties} in terms of Alexander duality. Let $K$
be the boundary of a simplicial $d$-polytope $P$ with $m$
vertices, thus a simplicial $(d-1)$-sphere. Then $\link_KI$ is a
simplicial (and even polytopal) $(d-1-|I|)$-sphere with $m-|I|$
vertices. By \eqref{eqAlexHomol} $\left(\link_KI\right)^{\wedge}$
should be a homological sphere of homological dimension $m-|I| -
(d-|I|-1) - 3 = m-d-2$. On the other hand, by
\eqref{eqLinkScAlex}, $\left(\link_KI\right)^{\wedge}$ is the same
as $\wh{K}_{[m]\setminus I}$. By claim \ref{claimGaleAlex},
$\wh{K}$ is the constellation complex for the nondegenerate
configuration on a sphere $\Ss^{m-d-2}$. These considerations
provide another explanation why full subcomplexes of a
nondegenerate constellation complex should be homology
$(m-d-2)$-spheres.

The proposition \ref{propResolution} can also be explained in
terms of Eagon--Reiner theorem \cite{EaRe}, which states that $K$
is Cohen--Macaulay if and only if Stanley--Reisner ideal of its
dual $\wh{K}$ has a linear resolution. For a simplicial polytope
$P$ the boundary $\partial P$ is a simplicial sphere, thus a
Cohen--Macaulay complex. Therefore $\Delta(G(P)) = (\partial
P)^{\wedge}$ has a linear resolution.

Moreover, there is a natural correspondence between simplices of
$K = \partial P$ and full subcomplexes of $\wh{K}$ homotopy
equivalent to $S^{m-d-2}$. This correspondence sends $I\in K$ to
the full subcomplex $\wh{K}_{[m]\setminus I}\simeq S^{m-d-2}$.
Thus, full subcomplexes, which are homotopy equivalent to a sphere
represent the face lattice of $K$.

\begin{prop}\label{propBettiConstel}
Let $P$ be a simplicial $d$-polytope with $m$ vertices, and
$X\subset\Ss^{m-d-2}$ --- its Gale diagram. Then for $i>0$ we have
$\beta^{-i}(\Delta(X))=\beta^{-i,m-d-1+i}(\Delta(X))=f_{d-i}(P)$,
where $f_{d-i}(P)$ is the number of $(d-i)$-dimensional simplices
of $\partial P$.
\end{prop}

\begin{proof}
We prove more general statement. Let $K$ be a simplicial
$(d-1)$-sphere. Then
$\beta^{-i}(\wh{K})=\beta^{-i,2(m-d-1+i)}(\wh{K})=f_{d-i}(K)$ ---
the number of $(d-i)$-dimensional simplices of $K$. The
proposition then follows by claim \ref{claimGaleAlex}.

For $i\neq 0$ we have:
\begin{multline}\label{eqBettiLink}
\beta^{-i,2j}(\wh{K})=\sum\limits_{|J|=j}\dim
\Hr^{j-i-1}(\wh{K}_J)=\sum\limits_{|J|=j}\dim
\Hr^{j-i-1}((\link_K([m]\setminus J))^{\wedge})=\\
=
\sum\limits_{|J|=j}\dim\Hr_{m-(m-j)-(j-i-1)-3}(\link_K([m]\setminus
J))=\sum\limits_{|J|=j}\dim \Hr_{i-2}(\link_K([m]\setminus J))=\\
=\sum\limits_{|J|=j}\dim \Hr_{i-2}(S^{d-1-m+j}).
\end{multline}
Thus $\beta^{-i,2j}(\wh{K})=0$ if $i-2\neq d-1-m+j$, that is
$j\neq m-d+i-1$. On the other hand, $\beta^{-i,m-d+i-1}(\wh{K})$
equals the number of simplices in $K$ with $m-j = m-(m-d+i-1) =
d-i+1$ vertices. Each such simplex contributes $1$ in the last sum
of \eqref{eqBettiLink}.
\end{proof}

To extend this result to general polytopes we use some basic facts
from the theory of nerve-complexes developed in \cite{AB}. Let $P$
be a $d$-dimensional polytope with $m$ vertices (possibly not
simplicial). Consider the numbers $f_{n,l}(P)$ --- the number of
$n$-dimensional proper faces of $P$ with $l$ vertices. In
addition, set $f_{-1,0}(P)=1$ --- this corresponds to the ``empty
face'' of a polytope.

Obviously, in the case when $P$ is simplicial we have
$f_{n,n+1}(P)=f_n(P)$ and $f_{n,l}(P)=0$ if $l\neq n+1$.
Generally, numbers $f_{n,l}$ provide much more detailed
information on a polytope, than the ordinary $f$-vector. In the
work \cite{AB} numbers $f_{n,l}$ appear as the coefficients of the
so called 2-dimensional $F$-polynomial of a polytope (it was
defined for the dual polytope, so the definition in that work is
slightly different).

\begin{prop}\label{propBettiConstelGen}
Let $P$ be a $d$-polytope with $m$ vertices, and
$X\subset\Ss^{m-d-2}\sqcup\{0\}$ --- its Gale diagram. Then for
$i>0$ we have $\beta^{-i,2j}(\Delta(X))=f_{d-i,m-j}(P)$.
\end{prop}

\begin{proof}
The calculation from proposition \ref{propBettiConstel} gives
\begin{equation}
\beta^{-i,2j}(\wh{K})=\sum\limits_{|J|=j}\dim
\Hr_{i-2}(\link_K([m]\setminus J)).
\end{equation}
In the case $K=K(P)$ it gives
\begin{equation}\label{eqBettiLink2}
\beta^{-i,2j}(\Delta(X))=\sum\limits_{|J|=j}\dim
\Hr_{i-2}(\link_{K(P)}([m]\setminus J)).
\end{equation}
For each face $F\subset P$ denote by $\ts(F)\subset K(P)$ the set
of its vertices. In \cite[Lemma 4.7]{AB} we proved that
$\link_{K(P)}\ts(F)\simeq S^{d - \dim F - 2}$ for any face
$F\subset P$ and for all other simplices $I\in K(P)$ the complex
$\link_{K(P)}I$ is contractible. Therefore, each face $F$ with
$\dim F = n$ and $|\ts(F)|=l$ contributes $1$ to the sum in
formula \eqref{eqBettiLink2} iff $l=m-j$ and $i-2=d-2-n$.
Therefore, $\beta^{-i,2j}(\Delta(X)) = f_{d-i,m-j}$.
\end{proof}

\begin{figure}[h]
\begin{center}
\includegraphics[scale=0.35]{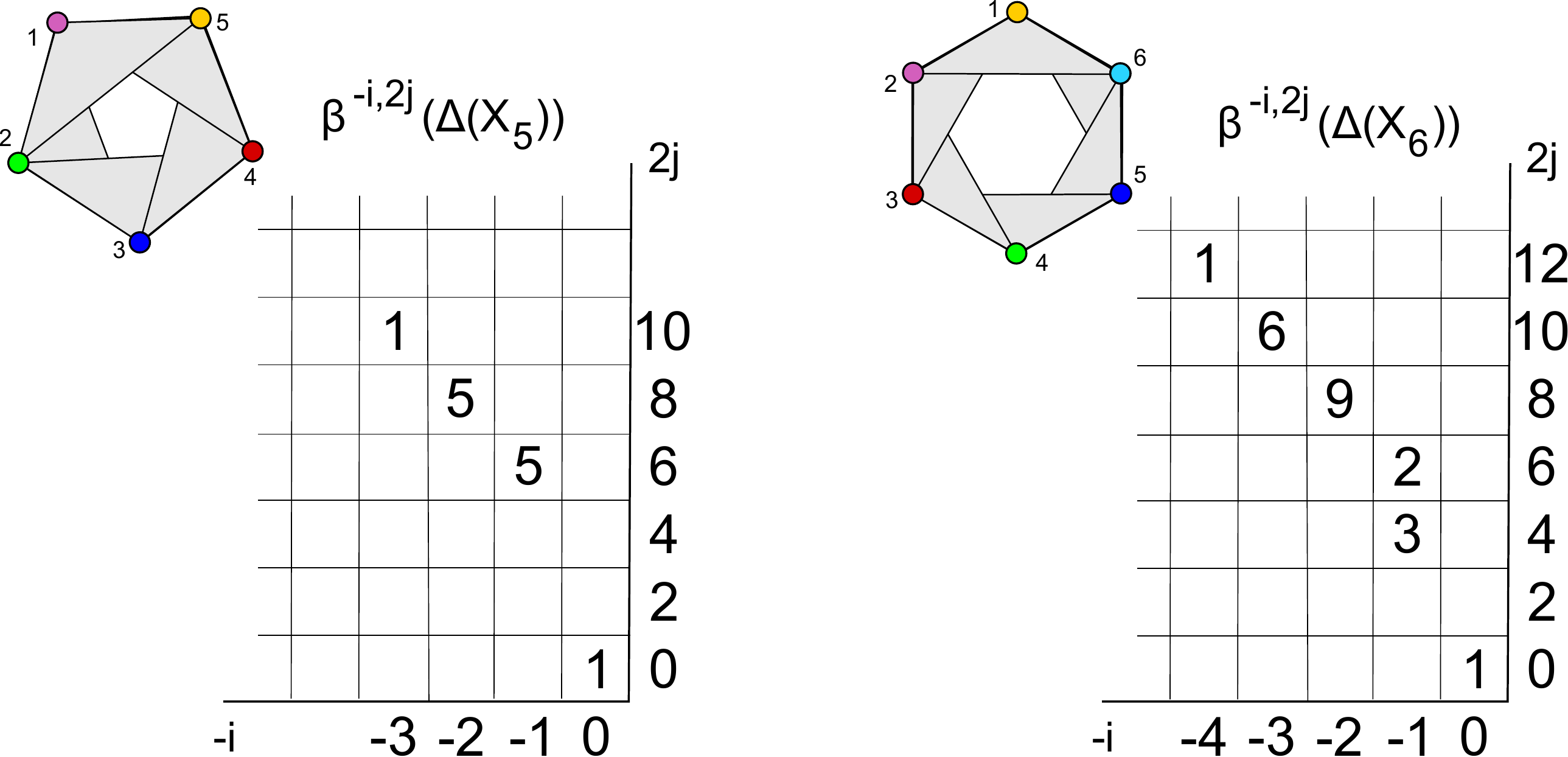}
\end{center}
\caption{Bigraded Betti numbers of complexes $\Delta(X_5)$ and
$\Delta(X_6)$}\label{pictBetti56}
\end{figure}

\ex Bigraded Betti numbers of $\Delta(X_5)$
(fig.\ref{pictCircle5}) and $\Delta(X_6)$ (fig. \ref{pictCircle6})
are depicted on fig. \ref{pictBetti56}. One can see, that bigraded
Betti numbers for $X_5$ are concentrated only in one dimension for
each $i$. This illustrates proposition \ref{propBettiConstel},
since $X_5$ is nondegenerate. The numbers in the left table
represent the $f$-vector of a pentagon, since $X_5$ is its Gale
diagram. The numbers from the right table represent $f_{n,l}$ for
a triangular prism (fig.\ref{pictPrismNerve}), since $X_6$ is its
Gale diagram, as was discussed in example \ref{ex6prism}. Indeed,
a triangular prism has $2$ $2$-dimensional triangular faces, $3$
$2$-dimensional quadratic faces, $9$ edges, $6$ vertices and $1$
empty face.

\subsection{Bigraded Betti numbers of polytopes}

We can use duality between subcomplexes and links in another
direction. Bigraded Betti numbers of $\partial P$ are described by
homology of links in $\Delta(X)$, where $X$ is the Gale diagram of
$P$.

\begin{prop}\label{propBettiOfPolytope}
Let $P$ be a simplicial polytope and $X$ --- its Gale diagram.
Then for $i>0$ there holds
$$\beta^{-i,2j}(K(P)) = \sum\limits_{J\in\Delta(X), |J|=m-j}\dim\Hr_{i-2}(\link_{\Delta(X)}J).$$
\end{prop}

The proof is similar to proposition \ref{propBettiConstel}. This
calculation can be used to count bigraded Betti numbers for
polytopes with small $m-d$, where $m$ is the number of vertices
and $d$ --- the dimension.

\begin{figure}[h]
\begin{center}
\includegraphics[scale=0.15]{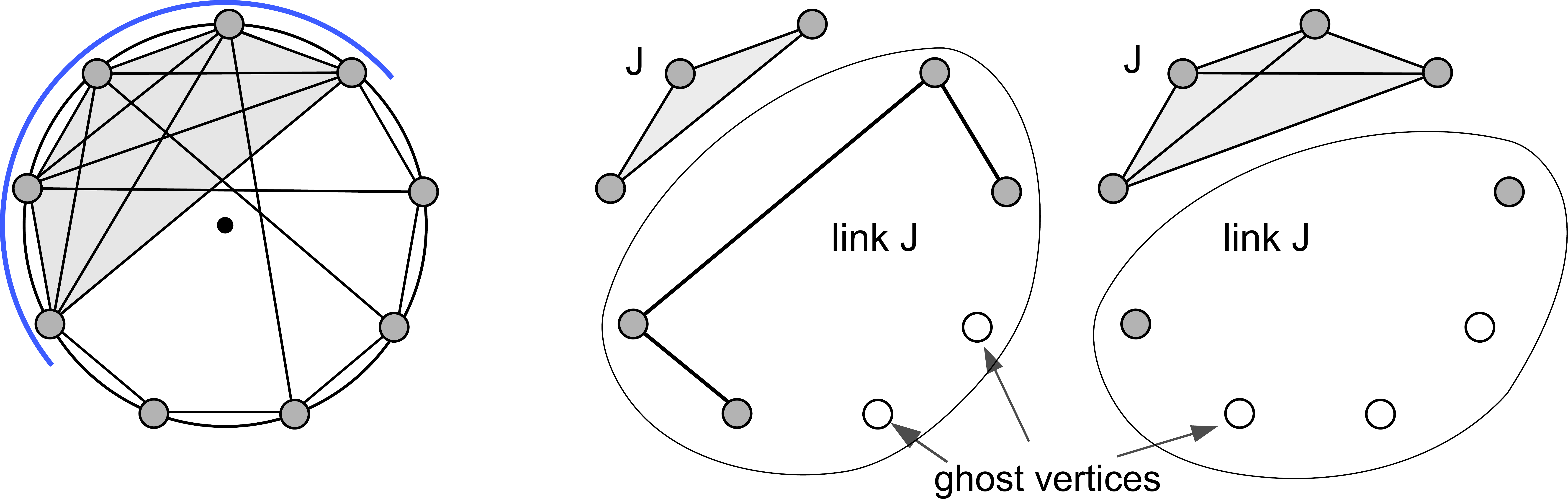}
\end{center}
\caption{Constellation complex for the regular $9$-gon and the
links of its simplices.}\label{pict9gon}
\end{figure}

\ex Let $P$ be a polytope with $m-d=3$. Its Gale diagram is up to
combinatorial equivalence a set $X\subset \Ss^1$ of points placed
in vertices of a regular $(2k+1)$-gon \cite[Sect.6.3]{Gr}. For
simplicity assume that $k\geqslant 2$ and any vertex of $X$
appears with multiplicity $1$, so $m=2k+1$. The case $k=2$ is
shown on fig.\ref{pictCircle5} and $k=4$ on fig.\ref{pict9gon}.
Any $k+1$ consecutive points of the configuration $X\subset \Ss^1$
lie in a common halfcircle, thus form a simplex of $\Delta(X)$.

The links of simplices $I\in\Delta(X)$ have nontrivial reduced
homology only in the following cases:
\begin{enumerate}
\item $I$ is the maximal simplex, that is a set of $k+1$
consecutive points of $X$. $\link_{\Delta(X)}I = \varnothing$ up
to ghost vertices and $\dim \Hr_{-1}(\link_{\Delta(X)}I)=1$.

\item $I$ is the set of $k$ consecutive points of $X$. In this
case $\link_{\Delta(X)}I = S^0$ up to ghost vertices and $\dim
\Hr_{0}(\link_{\Delta(X)}I)=1$.

\item $I=\varnothing$, $\link_{\Delta(X)}I = \Delta(X)$ and
$\dim \Hr_{1}(\link_{\Delta(X)}I)=1$.
\end{enumerate}
One can simply prove that for other choices of $I$ the complex
$\link_{\Delta(X)}I$ is contractible (see fig. \ref{pict9gon}).

By proposition \ref{propBettiOfPolytope} for the corresponding
simplicial polytope $P$ we have the expressions:
$\beta^{-1,2k}(K(P)) = 2k+1$, $\beta^{-2,2(k+1)}(K(P)) = 2k+1$,
$\beta^{-3,2(2k+1)}(K(P))=1$ and $\beta^{-i,2j}(K(P))=0$ for all
other $i,j$ with $i\neq 0$.

The same computation can be made if nontrivial multiplicities are
assigned to vertices of the Gale diagram $X\subset \Ss^1$, but the
answer is more sophisticated. The answer coincides with the result
of \cite{ErArx}, where bigraded Betti numbers of simplicial
spheres with few vertices were calculated using different
approach.

\section{Buchstaber invariant}

\begin{defin}
Let $K$ be a simplicial complex on $m$ vertices and $D^2,S^1$ be
the unit disk and unit circle in $\Co$. For any simplex $I\in K$
define the subset $(D^2,S^1)^I\subset (D^2)^m$, $(D^2,S^1)^I =
\{(x_1,\ldots,x_m)\in (D^2)^m, x_i\in S^1$, if $i\notin I\}$. Then
the \emph{moment-angle complex} of $K$ is a topological space
$$
\Z_K = \bigcup\limits_{I\in K}(D^2,S^1)^I \subseteq (D^2)^m.
$$
\end{defin}

A \emph{real moment-angle complex} $\Zr_K$ is defined in the same
manner by using $D^1$ and $S^0$ as building blocks instead of
$D^2$ and $S^1$ respectively \cite{BP,BPnew}.

There is a canonical coordinatewise action of a torus $T^m =
S^1\times\ldots\times S^1$ on $(D^2)^m$, which can be restricted
to the action on a subset $T^m\curvearrowright\Z_K$. Similarly, we
have an action $\Zt^m{\curvearrowright} {\Zr_K}$. In both cases
the action is not free for $K\neq\varnothing$.

\begin{defin}\label{definBuchNumber}
\emph{Buchstaber invariant} $s(K)$ of the complex $K$ is the
maximal rank of torus subgroups $G\subset T^m$ which act freely on
$\Z_K$. Real Buchstaber invariant $\sr(K)$ is the maximal rank of
subgroups $G\subset \Zt^m$ which act freely on $\Zr_K$.
\end{defin}

We briefly sketch here the main consideration, which allows to
study Buchstaber invariants. The action $T^m\curvearrowright\Z_K$
have stabilizer subgroups of the form $T^I\subseteq T^m$ --- the
coordinate subtori, corresponding to simplices $I\in K$. The
subgroup $G\subseteq T^m$ acts freely on $\Z_K$ iff it intersects
stabilizers trivially.

By dimensional reasons, $\rk G \leqslant m - \max \rk(T^I) =
m-\dim K-1$. Therefore, $s(K)\leqslant m-\dim K - 1$. Similarly
$\sr(K)\leqslant m-\dim K -1$. If $K\neq \Delta_{[m]}$, then
$s(K)\geqslant 1$ and $\sr(K)\geqslant 1$, since diagonal
subgroups act freely. By taking real part of a moment-angle
complex one can also prove that $\sr(K)\geqslant s(K)$. So, we
have:
\begin{equation}\label{eqBuchMainEq}
1\leqslant s(K)\leqslant \sr(K)\leqslant m-\dim K -1.
\end{equation}

\begin{prop}\label{propBuchComplexRealEqual}
If one of the following conditions holds, then $\sr(K)=s(K)$:
\begin{enumerate}
\item $s(K)=m-\dim K-1$ (follows from \eqref{eqBuchMainEq}),

\item $\dim K = 0, 1$ or $2$ \cite{Ayzs,ErNewBig}

\item Either $s(K)$ or $\sr(K)$ is equal to $1$ or $2$.
\cite{ErArx}.
\end{enumerate}
\end{prop}

\subsection{Case of pyramids}

Recall from \S\ref{SecSimpConstr}, that $K_P = K(P^*)$. Polytope
$P$ is simple whenever $P^*$ is simplicial.

\begin{defin}\label{definBuchPoly}
For a (possibly nonsimple) polytope $P$ define $s(P) = s(K_P)$ and
$\sr(P) = \sr(K_P)$.
\end{defin}

Suppose, $P = \pyr Q$ is a pyramid with $(d-1)$-dimensional
polytope $Q$ in the base. In this case $P^* = \pyr Q^*$. So $P$ is
a pyramid whenever $P^*$ is a pyramid.

\begin{thm}\label{thmPyramidS}
Let $P$ be a polytope. Then $s(P)=1$ if and only if $P$ is a
pyramid.
\end{thm}

\begin{proof}
Let $\dim P=d$ and $P$ has $m$ facets.

The ``if'' part is simple. If $P$ is a pyramid with $m$ facets,
then $P^*$ is a pyramid with $m$ vertices. All of these vertices
except apex lie in a same facet --- in the base. These vertices by
definition form a simplex $I\in K(P^*)=K_P$, such that $|I|=m-1$.
Therefore, $\dim K(P^*)=m-2$ and the desired relation $s(P)=1$
follows from \eqref{eqBuchMainEq}.

Now the ``only if'' part. Suppose that $P$ is not a pyramid. Then
$P^*$ is not a pyramid and its Gale diagram $G(P^*)$ does not
contain points at $\{0\}$ by proposition \ref{propPyrZero}. By
claim \ref{claimGaleAlex} the complex $K(P^*)^{\wedge}$ is a
constellation complex for the configuration $G(P^*)\subset
\Ss^{m-d-2}$. It does not have ghost vertices.

Let $n\in\Ss^{m-d-2}$ be such a vector, that $\langle n,
x_i\rangle\neq 0$ for each point $x_i\in G(P^*)$. Then every point
from $G(P^*)$ lies either in $H(n)$ or $H(-n)$ (the open
hemispheres introduced in \S\ref{SecSimpConstr}). Let $I_{+}$ and
$I_{-}$ be the sets of points of Gale diagram sitting in $H(n)$
and $H(-n)$ respectively. These sets are by definition the
simplices of constellation complex $I_{+}, I_{-}\in
\Delta(G(P^*))$, and $I_{+}\sqcup I_{-} = [m]$.

Since $\Delta(G(P^*)) = K(P^*)^{\wedge}$, we have
$I_{-}=[m]\setminus I_{+}\notin K(P^*)$ and $I_{+}=[m]\setminus
I_{-}\notin K(P^*)$ by definition of Alexander dual complex. Also
$I_{+}\cap I_{-}=\varnothing$. $I_{+}$ and $I_{-}$ are
nonsimplices of $K(P^*)$ so they contain minimal nonsimplices
$J_1\subseteq I_{+}$ and $J_2\subseteq I_{-}$, $J_1,J_2\in
N(K(P^*))$, which, obviously, satisfy $J_1\cap J_2 = \varnothing$.

Now we use recent result of N.Yu.Erokhovets \cite{ErNew} which is
the following. Let $K$ be a simplicial complex. If $K$ has three
minimal nonsimplices or two minimal nonsimplices, which do not
intersect, then $s(K)\geqslant 2$. Applying this result to
$K(P^*)$ finishes the proof.
\end{proof}

\rem The similar theorem holds for real Buchstaber number by
statement \ref{propBuchComplexRealEqual}.

\rem If we restrict to the class of simple polytopes, the theorem
\ref{thmPyramidS} is known \cite{ErArx}: for simple polytope $P$
the condition $s(P)=1$ is equivalent to $P = \triangle^{m-1}$.
Obviously, a simplex is the only simple polytope, which is a
pyramid.

\subsection{General polytopes}

In this section we apply the result of \cite{ErNew} and Gale
duality to general polytopes. This leads to interesting
combinatorial consequences.

A set of vectors $\{a_1,\ldots,a_l\}\subset \Zt^k$ is called a
minimal linear dependence if $a_1+\ldots+a_l=0$, but its proper
subsets are linearly independent.

\begin{prop}[{\cite[Prop.9]{ErNew}}]\label{propKolyaNonsimpl}
$\sr(K)\geqslant k$ if and only if there exist a mapping\linebreak
$\xi\colon\Zt^k\setminus \{0\}\to N(K)$ such that
$\xi(a_1)\cap\ldots\cap\xi(a_{2r+1})=\varnothing$ for any minimal
linear dependence $\{a_1,\ldots,a_{2r+1}\}$ (with an odd number of
elements).
\end{prop}

\rem\label{remMapToNonsimplices} Equivalently, there exist a
mapping $\xi$ from $\Zt^k\setminus \{0\}$ to the set of all
nonsimplices (not necessary minimal), satisfying
$\xi(a_1)\cap\ldots\cap\xi(a_{2r+1})=\varnothing$ for any minimal
linear dependence $\{a_1,\ldots,a_{2r+1}\}$. Indeed, if there is
such a map, we can choose a minimal nonsimplex inside each
$\xi(a)$ and these subsets satisfy the same nonintersecting
condition.

\begin{cor}
For any simplicial complex $K$ on $[m]$ and an array
$(l_1,\ldots,l_m)$ of positive integers we have
$s(K(l_1,\ldots,l_m))=s(K)$.
\end{cor}

\begin{proof}
This follows from the description of minimal nonsimplices of
iterated wedge construction $K(l_1,\ldots,l_m)$ \cite{BBCGit} and
proposition \ref{propKolyaNonsimpl}.
\end{proof}

\begin{prop}\label{propErokhReform}
Let $P$ be a $d$-dimensional polytope with $m$ facets. Then
$\sr(P)\geqslant k$ if and only if there exist a map
$\eta\colon\Zt^k\setminus\{0\}\to \Ss^{m-d-2}$ such that
$G(P^*)\subset H(\eta(a_1))\cup\ldots\cup H(\eta(a_{2r+1}))$ for
any odd minimal linear dependence $\{a_1,\ldots,a_{2r+1}\}$ in
$\Zt^k$.
\end{prop}

\begin{proof}
Recall that $\sr(P) = \sr(K(P^*))$. By remark
\ref{remMapToNonsimplices} $\sr(K(P^*))\geqslant k$ is equivalent
to the existence of a map $\xi$ from $\Zt^k\setminus\{0\}$ to
nonsimplices of $K(P^*)$ satisfying certain conditions. The
complement to any nonsimplex of $K(P^*)$ is a simplex of
$K(P^*)^{\wedge}$, therefore we have a map $\xi'\colon
\Zt^k\setminus\{0\}\to K(P^*)^{\wedge}$,
$\xi'(a)=[m]\setminus\xi(a)$. Map $\xi'$ satisfies the condition
\begin{equation}\label{eqSimplexCover}
\xi'(a_1)\cup\ldots\cup\xi'(a_{2r+1})=[m]
\end{equation}
for any minimal linear dependence $\{a_1,\ldots,a_{2r+1}\}$.

For each binary vector $a\in \Zt^k\setminus\{0\}$ consider a
simplex $\xi'(a)\in K(P^*)^{\wedge}$. By claim \ref{claimGaleAlex}
$K(P^*)^{\wedge}=\Delta(G(P^*))$, so the points of $G(P^*)$ with
labels from $\xi'(a)$ lie in an open hemisphere. So, there exist a
vector $n_a\in \Ss^{m-d-2}$, such that the hemisphere $H(n_a)$
contains $G(P^*)(\xi'(a))$. Set $\eta(a) = n_a$.

If $\{a_1,\ldots,a_{2r+1}\}$ is a minimal linear dependence, then
every label $i\in [m]$ lies in some simplex $\xi'(a_q)$ by
\eqref{eqSimplexCover}. Thus any point of $G(P^*)$ lies in
$H(\eta(a_q))$ for some $q\in[2r+1]$, which was to be proved.

The proof goes the same in opposite direction.
\end{proof}

This gives an idea how to construct polytopes with
$\sr(P)\geqslant k$.

\begin{cor}\label{corConstrBuchGiven}
Consider an arbitrary map $\eta\colon \Zt^k\setminus\{0\}\to
\Ss^{l}$. Let $X=\{x_1,\ldots,x_m\}\subset\Ss^l$ be a spherical
configuration such that $X\subset H(\eta(a_1))\cup\ldots\cup
H(\eta(a_{2r+1}))$ for any odd minimal linear dependence
$\{a_1,\ldots,a_{2r+1}\}$ in $\Zt^k$. Suppose $\bigcup H(x_i)$
covers $\Ss^l$ at least twice, so $X=G(P)$ for some
$(m-l-2)$-dimensional polytope $P$ with $m$ vertices. Then
$\sr(P^*)\geqslant k$.
\end{cor}

This construction has an unexpected purely combinatorial
consequence. Let us call the coloring
$\rho\colon\Zt^k\setminus\{0\}\to C$ proper, if any odd minimal
affine dependence $A=\{a_1,\ldots,a_{2r+1}\}$ is not
single-colored, that is $|\rho(A)|\neq 1$.

\begin{thm}\label{thmColorZt}\mbox{}
\begin{enumerate}
\item There is no proper coloring of $\Zt^k\setminus\{0\}$ by $k-1$
colors.
\item There exist a proper coloring of $\Zt^k\setminus\{0\}$ by $k$
colors.
\end{enumerate}
\end{thm}

\begin{proof}
(1) The statement, obviously, holds for $k=1,2$, so in the
following let $k\geqslant 3$. Suppose the contrary. Let
$\rho\colon \Zt^k\setminus\{0\} \to [k-1]$ be a proper coloring.
Consider a set of points
$\{y_1,\ldots,y_{k-1}\}\subset\Ss^{k-3}\subset \Ro^{k-2}$,
representing the vertices of a regular simplex, inscribed in
$\Ss^{k-3}$. We have $\sum y_i=0$, and $\langle
y_i,y_j\rangle=-\frac{1}{k-2}<0$ for $i\neq j$. Take antipodal
points $x_i=-y_i$ for $i\in[k-1]$. Then

\begin{enumerate}
\item $H(y_i)\subset \Ss^{k-3}$ contains all $x_j$ for $j\neq i$.
\item $0\in \relint\conv\{x_1,\ldots,x_{k-1}\}$.
\end{enumerate}

Consider a map $\eta\colon\Zt^k\setminus\{0\}\to \Ss^{k-3}$,
defined by $\eta(a)=y_{\rho(a)}$. This map associates to binary
vector $a\in \Zt^k\setminus\{0\}$ one of the points $y_i$ on a
sphere according to the coloring. Since $\rho$ is a proper
coloring, for any minimal affine dependence
$A=\{a_1,\ldots,a_{2r+1}\}$ the set $\eta(A)$ contains at least
two points $y_i$ and $y_j$, $i\neq j$. By (1), the union of
hemispheres $H(y_i)\cup H(y_j)$ covers the set $\{x_1,\ldots,
x_{k-1}\}$.

Now take each point $x_1,\ldots, x_{k-1}$ with multiplicity at
least $2$ to get a configuration $X$ of $m$ points on $\Ss^{k-3}$.
Configuration $X$ is good and nondegenerate according to (2)
(recall, that good means that corresponding hemispheres cover
$\Ss^{k-3}$ at least twice). Therefore, by claim
\ref{claimGaleAlex} there exist a simple polytope $P$ of dimension
$m-(k-3)-2=m-k+1$ with $m$ facets, such that $G(P^*)=X$. Then, by
corollary \ref{corConstrBuchGiven} $\sr(P)\geqslant k$. On the
other hand, by estimation \eqref{eqBuchMainEq} we have
$\sr(P)\leqslant m-\dim P = m-(m-k+1)=k-1$. This gives a
contradiction.

(2) Let $\{e_1,\ldots,e_k\}$ be the basis of a space $\Ro^k$.
Consider the crosspolytope\linebreak $P=\conv\{\pm e_1,\ldots,\pm
e_k\}$. It is simplicial and has $m=2k$ vertices, so its Gale
diagram $G(P)$ lies on a sphere $\Ss^{k-2}$. One can show, that
$G(P)=\{x_1,x_1,x_2,x_2,\ldots,x_k,x_k\}$ --- the vertices of a
regular simplex, inscribed in $\Ss^{k-2}$, each taken two times.

The dual polytope $P^*$ is a cube. One can easily derive from
\cite[Th. item 4]{ErArx} (or from the fact that $P^*$ is Delzant)
that $\sr(P^*)=m-\dim P = k$. Then by statement
\ref{propErokhReform} there exist a map from
$\eta\colon\Zt^k\setminus \{0\}\to \Ss^{k-2}$, such that
$H(\eta(a_1))\cup\ldots\cup H(\eta(a_{2r+1}))$ contain
$G(P)\subset\Ss^{r-2}$ for each odd minimal linear dependence
$\{a_1,\ldots,a_{2r+1}\}$. To each $a\in \Zt^k\setminus\{0\}$
assign a color $\rho(a)=j\in [k]$ if $H(\eta(a))$ does not contain
$x_j$. Then to each $a$ at least one color is assigned, since
$x_1,\ldots,x_k$ do not lie in a common hemisphere. If several
colors are assigned to binary vector $a$, choose any of them. We
claim that coloring $\rho\colon \Zt^k\to[k]$ obtained by this
procedure is proper. Suppose the contrary: $\rho(A)=\{j\}$ for
some minimal affine dependence $A=\{a_1,\ldots,a_{2r+1}\}$. In
this case all hemispheres $H(\eta(a_i))$ do not contain $x_j$
which contradicts the construction.
\end{proof}

Recall, that Fano plane is a finite projective geometry
$\mathbb{P}\Zt^3$. Its points are nonzero binary vectors in
$\Zt^3$ and lines are triples $\{a_1,a_2,a_1+a_2\}$, $a_1\neq a_2$
that is exactly minimal linear dependencies.

\begin{cor}
For any coloring of Fano plane by two colors there exist a line of
single color.
\end{cor}

\rem We would like to mention that the technic used to prove
theorem \ref{thmColorZt} looks very similar to the proof of Kneser
conjecture, found by B\'{a}r\'{a}ny \cite{Bar}. Though a direct
connection of these subjects is not clarified yet.
\\
\\
Here is another unexpected fact, derived from the theory of
Buchstaber invariant.

\begin{prop}\label{propFanoCircle}
For any map $\eta\colon \mathbb{P}\Zt^3\to \Ss^1$ there exist a
Fano line $\{a,b,c\}\subset \mathbb{P}\Zt^3$ such that
$0\notin\conv\{\eta(a),\eta(b),\eta(c)\}$.
\end{prop}

\begin{figure}[h]
\begin{center}
\includegraphics[scale=0.2]{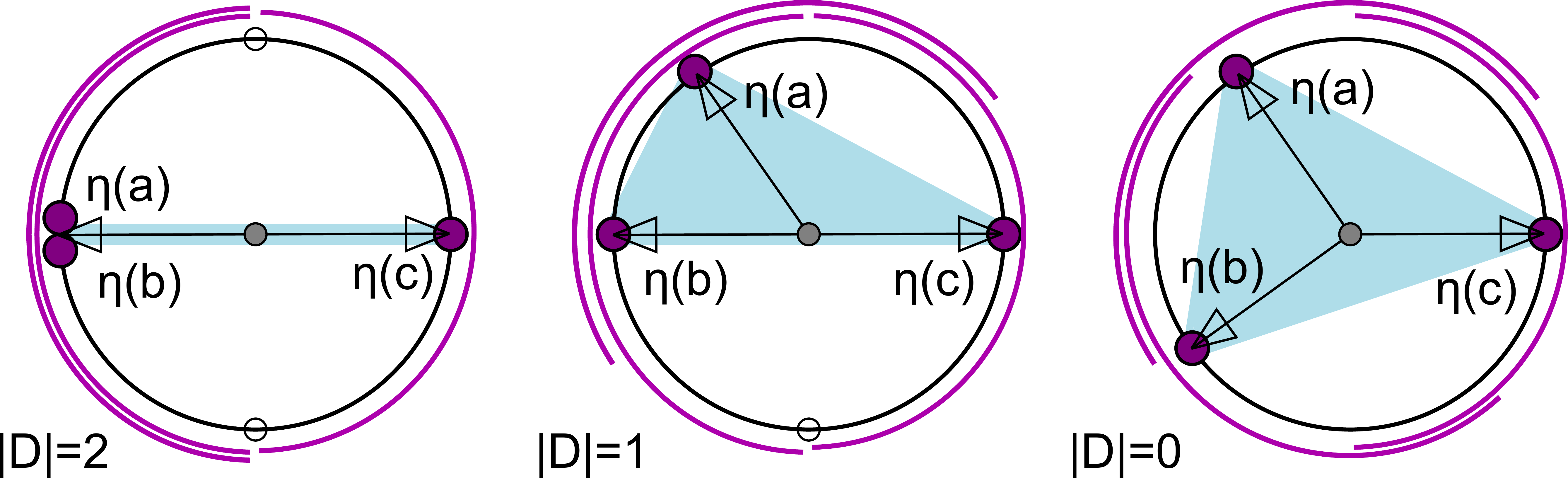}
\end{center}
\caption{Possible cases for $0\in
\conv\{\eta(a),\eta(b),\eta(c)\}$}\label{pictCircle3}
\end{figure}

\begin{proof}
Suppose the contrary. Then, for each Fano line $l=\{a,b,c\}$ holds
either $0\in \conv\{\eta(a),\eta(b)\}$ (or the same for other pair
of points); or $0\in \relint\conv \{\eta(a),\eta(b),\eta(c)\}$. In
both cases $H(\eta(a))\cup H(\eta(b))\cup H(\eta(c)) =
\Ss^1\setminus D_l$, where $D_l$ is a finite set (see fig.
\ref{pictCircle3}), $|D_l|\leqslant 2$. Now, let $X$ be an
arbitrary good nondegenerate configuration of points on a circle
$\Ss^1$. Then $X=G(P^*)$ for a simple polytope $P$. By slightly
rotating $X$ we can assume that $X$ does not intersect the finite
set $\bigcup_lD_l$. Then $X$ is covered by $H(\eta(a))\cup
H(\eta(b))\cup H(\eta(c))$ for any Fano line $\{a,b,c\}$. Then, by
corollary \ref{corConstrBuchGiven} $\sr(P)\geqslant 3$. This
consideration shows that $\sr(P)\geqslant 3$ for any simple
polytope $P$ with $m$ facets of dimension $m-3$. This would
contradict the following result of \cite{Er,ErArx}: for each $k>2$
there exist a simple polytope $P$ with $m$ facets, such that
$m-\dim P=k$ and $\sr(P)=2$.
\end{proof}


\begin{thebibliography}{99}

\bibitem{Ayzs} A. Ayzenberg, \textit{Relation between the Buchstaber invariant and
generalized chromatic numbers}, Far-Eastern Math.
J.,11:2(2011),113-139. (Anton Ayzenberg, \textit{The problem of
Buchstaber number and its combinatorial aspects}, arXiv:1003.0637
[math.CO])

\bibitem{Ayzit} Anton Ayzenberg, \textit{Composition of simplicial complexes,
polytopes and multigraded Betti numbers}, arXiv:1301.4459
[math.CO]

\bibitem{AB} A. A. Ayzenberg, V. M. Buchstaber, \textit{Moment-angle spaces and
nerve-complexes of convex polytopes}, Proceedings of the Steklov
Institute of Mathematics, V.275, 2011.

\bibitem{ABarx} A.A.Ayzenberg, V.M.Buchstaber, \textit{Moment-angle complexes
and polyhedral products for convex polytopes}, arXiv:1010.1922
[math.CO]

\bibitem{BBCGit} A.\,Bahri, M.\,Bendersky, F.\,R.\,Cohen,
S.\,Gitler, \textit{Operations on polyhedral products and a new
topological construction of infinite families of toric manifolds},
arXiv:1011.0094v4 [math.AT]

\bibitem{Bar} J B\'{a}r\'{a}ny, \textit{A short proof of Kneser's
conjecture}, Journal of Combinatorial Theory, Series A, Vol. 25,
Issue 3, 1978, pp.325--326.

\bibitem{BP} V.\,M.\,Buchstaber and T.\,E.\,Panov, Torus Actions and Their
Applications in Topology and Combinatorics // University Lecture,
vol. 24, Amer. Math. Soc., Providence, R.I., 2002.

\bibitem{BPnew} Victor Buchstaber, Taras Panov, \textit{Toric
Topology}, arXiv:1210.2368 [math.AT]

\bibitem{CS} J.H.Conway, N.J.A.Sloane, Sphere packings, lattices
and groups // 3rd ed., A Series of Comprehensive Studies in
Mathematics, Vol. 290, 1999.

\bibitem{EaRe} John A. Eagon and Victor Reiner, \textit{Resolutions of
Stanley--Reisner Rings and Alexander Duality}, J. Pure Appl.
Algebra 130 (1998), N.3, 265--275.

\bibitem{Er} Nikolai Yu Erokhovets, \textit{Buchstaber invariant of simple
polytopes}, Russian Mathematical Surveys(2008),63(5):962

\bibitem{ErArx} Nickolai Erokhovets, \textit{Buchstaber Invariant of Simple
Polytopes}, arXiv:0908.3407 [math.AT]

\bibitem{ErThes} Nickolai Erokhovets, \textit{Maximal torus actions on
moment-angle manifolds}, doctoral thesis, Moscow State University,
Faculty of Mechanics and Mathematics, 2011. (in russian).

\bibitem{ErNew} Nickolai Erokhovets, \textit{Criterion for the Buchstaber
invariant of simplicial complexes to be equal to two},
arXiv:1212.3970 [math.AT]

\bibitem{ErNewBig} Nickolai Erokhovets, \textit{The theory of Buchstaber
invariant of simplicial complexes and convex polytopes}, draft, to
appear.

\bibitem{FM} Yukiko Fukukawa and Mikiya Masuda,
\textit{Buchstaber invariants of skeleta of a simplex}, Osaka J.
Math. V. 48, N.2 (2011), 549--582; arXiv:0908.3448v2.

\bibitem{Gal} Swiatoslaw R. Gal, \textit{Real Root Conjecture
fails for five and higher dimensional spheres}, Discrete and
Computational Geometry, Vol. 34, Number 2, pp. 269-284, 2005,
arXiv:math/0501046 [math.CO]

\bibitem{Gale} D.\,Gale, \textit{Neighboring vertices on a convex
polyhedron}, in Linear inequalities and related systems, edited by
H.W.Kuhn and A.W.Tucker, Princeton, 1956.

\bibitem{Gr} Branko Gr\"{u}nbaum, Convex Polytopes // 2nd ed.,
Graduate Texts in Mathematics Vol. 221, 2003.

\bibitem{Hoch}
M.\,Hochster, \textit{Cohen-Macaulay rings, combinatorics, and
simplicial complexes}, in Ring theory, II (Proc. Second
Conf.,Univ. Oklahoma, Norman, Okla., 1975),  Lecture Notes in Pure
and Appl. Math., V. 26, pp.171--223, Dekker, New York, 1977.

\bibitem{Izm1} I.\,V.\,Izmestiev, \textit{Three-Dimensional
Manifolds Defined by Coloring a Simple Polytope}, Mathematical
Notes March 2001, Volume 69, Issue 3-4, pp. 340--346.

\bibitem{Izm2} I.\,V.\,Izmest'ev, \textit{Free torus action on
the manifold $Z_P$ and the group of projectivities of a polytope
$P$}, Russian Mathematical Surveys(2001),56(3):582.

\bibitem{Marc} Daniel A.\,Marcus, \textit{Gale diagrams of convex polytopes
and positive spanning sets of vectors}, Discrete Applied
Mathematics, Vol. 9, Issue 1, 1984, pp. 47--67.


\end{thebibliography}
\end{document}